\documentclass{svproc}

\usepackage{amsmath,amssymb,amsfonts}
\usepackage{graphicx}
\usepackage{url}
\usepackage{bm}
\usepackage{mathrsfs}
\usepackage{cite}
\usepackage{hyperref}
\usepackage{cleveref}
\usepackage{doi}

\numberwithin{theorem}{section}
\numberwithin{lemma}{section}
\numberwithin{proposition}{section}
\numberwithin{corollary}{section}
\numberwithin{definition}{section}
\numberwithin{remark}{section}

\begin{document}

\title{Threshold Hierarchy for Packet-Scale Boundary Cancellation of Dirichlet Eigenfunctions}
\titlerunning{Threshold Hierarchy for Boundary Cancellation}

\author{Anton Alexa}
\authorrunning{Anton Alexa}
\tocauthor{Anton Alexa}

\institute{Independent Researcher, Chernivtsi, Ukraine\\
\email{mail@antonalexa.com}}

\maketitle

\begin{abstract}
We identify geometry--dependent minimal packet scales required for cancellation
of boundary correlations of high--frequency Dirichlet eigenfunctions on smooth
strictly convex domains.
The main result is a threshold hierarchy: for zero--mean boundary weights,
the energy--weighted packet average of boundary correlation coefficients
vanishes once the packet length exceeds a scale determined by the vanishing
order of curvature moments of the weight.
In particular, the threshold $N_k/k^{1-2/d}\to\infty$ suffices when
$\int_{\partial\Omega} w\,d\sigma=0$, while a strictly weaker threshold applies
when additionally $\int_{\partial\Omega} H\,w\,d\sigma=0$, reducing in
dimension $d=3$ to the minimal condition $N_k\to\infty$.
The thresholds follow from the boundary local Weyl law.
As a structural consequence of the Rellich identity alone, the single--mode
share of boundary energy within any sublinear spectral packet is of order
$1/N_k$.
All estimates are independent of eigenvalue monotonicity and remain stable
under eigenvalue crossings.
\end{abstract}

\section{Introduction}
Let $\Omega \subset \mathbb{R}^d$ be a bounded $C^\infty$ strictly convex domain,
$d \ge 2$, and let $\{(\lambda_k,u_k)\}_{k\ge 1}$ denote the Dirichlet eigenpairs of
$-\Delta_\Omega$ with $\|u_k\|_{L^2(\Omega)}=1$. The boundary flux density
$|\partial_n u_k|^2$ plays a central role in boundary spectral analysis and in
microlocal restriction phenomena. A natural question is whether a single
high--frequency eigenfunction can exhibit dominant boundary localization compared
to its spectral neighbors.

Related results on boundary delocalization and restriction estimates typically rely
on microlocal analysis, defect measures, or dynamical properties of the billiard flow \cite{BGT,HassellTacy,SoggeZelditch,Grieser,TothZelditch}.
For general background on boundary eigenfunctions and related phenomena,
see \cite{HassellSurvey,ZelditchSurvey}.
Classical boundary quantum ergodicity and boundary-trace completeness results
include \cite{GerardLeichtnam,HassellZelditchQEBoundary,BurqQEBoundary,HanHassellHezariZelditch}.
For recent broader developments in related semiclassical restriction and
eigenfunction analysis, see \cite{ChristiansonTothFlux,CanzaniGalkowskiBeams}.
In contrast, our approach is entirely static and relies on global spectral
identities and boundary local Weyl asymptotics: the threshold hierarchy is
established via the boundary local Weyl law
\cite{SafarovVassiliev,Ivrii,BransonGilkey}, while the mode--to--packet
comparison follows from the Rellich identity alone, with no microlocal input.
Our approach is complementary in spirit to microlocal and dynamical studies of
boundary eigenfunctions
\cite{BGT,HassellTacy,SoggeZelditch,Grieser,TothZelditch,HassellSurvey,ZelditchSurvey}.
The primary motivation for this packet--level viewpoint is structural stability
under spectral deformations. In settings with shape perturbations or geometric
flows, eigenvalue crossings and changing multiplicities are typical. Since the
estimates below are formulated via spectral sums, they are invariant under
changes of orthonormal bases inside eigenspaces and remain stable across such
crossings.
Moreover, all constants are uniform over compact $C^\infty$ families of
strictly convex domains (see Remark~\ref{rem:uniform-rate}), making the
threshold hierarchy directly applicable in dynamical settings where the domain
evolves continuously and cancellation must be controlled uniformly in time.

The principal new contribution of this paper is the identification of
geometry--dependent minimal packet scales for boundary correlation cancellation
on strictly convex domains.
For any zero--mean weight $w\in C^\infty(\partial\Omega)$, the energy--weighted
packet average of boundary correlation coefficients vanishes once $N_k$ crosses
a threshold determined by the vanishing order of curvature moments of $w$,
in the spirit of averaged spectral asymptotics \cite{ColindeVerdiere,LaptevSafarov}.
This threshold forms a hierarchy of two levels:
the first, $N_k/k^{1-2/d}\to\infty$, suffices when $\int_{\partial\Omega} w\,d\sigma=0$;
the second, strictly weaker level $N_k/k^{(d-3)/d}\to\infty$ applies when
additionally $\int_{\partial\Omega} H\,w\,d\sigma=0$,
reducing for $d=3$ to the minimal condition $N_k\to\infty$.
The hierarchy is stratified by the vanishing order of spectral moments of $w$
with respect to boundary geometry and is a new structural feature of the
boundary spectral problem with no analogue in bulk delocalization results.
We emphasize that the novelty lies not in the boundary Weyl expansion itself,
which is classical \cite{SafarovVassiliev,Ivrii,BransonGilkey}, but in the
extraction of \emph{explicit minimal packet-length scales} forced by
successive moment cancellations and the resulting detectability hierarchy:
the question "how much spectral averaging is needed to make boundary correlations
undetectable by the Rellich--Weyl mechanism?" is answered quantitatively and
stratified by the geometry of $w$.

\begin{theorem}[Threshold hierarchy for boundary cancellation --- Introductory Form]
\label{thm:threshold-intro}
Let $\Omega \subset \mathbb{R}^d$ be a bounded $C^\infty$ strictly convex domain,
$d \ge 2$, and let $\{(\lambda_k,u_k)\}_{k\ge1}$ be the Dirichlet eigenpairs of
$-\Delta_\Omega$.
Let $w \in C^\infty(\partial\Omega)$ and let $N_k$ satisfy $N_k\to\infty$ and $N_k=o(k)$.
\begin{enumerate}
\item[(i)] If\/ $\int_{\partial\Omega} w\,d\sigma = 0$ and $N_k/k^{1-2/d}\to\infty$,
then
\begin{equation}
\frac{\displaystyle\sum_{j=k}^{k+N_k-1}
\int_{\partial\Omega}|\partial_n u_j|^2\, w\, d\sigma}
{\displaystyle\sum_{j=k}^{k+N_k-1} E_j}
\longrightarrow 0
\quad\text{as }k\to\infty.
\end{equation}
\item[(ii)] If additionally $\int_{\partial\Omega} H\,w\,d\sigma = 0$ and
$N_k/k^{(d-3)/d}\to\infty$, the same conclusion holds.
For $d=3$, case~(ii) holds under the minimal condition $N_k\to\infty$.
\end{enumerate}
These thresholds are geometry--dependent: they are determined by the vanishing
order of curvature moments of $w$ with respect to the boundary of\/ $\Omega$.
\end{theorem}

We next record a structural comparison result following directly from the
Rellich identity, which explains the intrinsic $1/N_k$ scaling of single--mode
contributions within spectral packets and does not require Weyl asymptotics or
microlocal analysis.

\begin{theorem}[Structural consequence: mode-to-packet comparison --- Introductory Form]
\label{thm:intro}
Let $\Omega \subset \mathbb{R}^d$ be a bounded $C^\infty$ strictly convex domain,
$d \ge 2$, and let $\{(\lambda_k,u_k)\}_{k\ge1}$ be the Dirichlet eigenpairs of
$-\Delta_\Omega$.
Let $N_k$ be any sublinear spectral window with $N_k\to\infty$ and $N_k=o(k)$.
Then there exist constants $0<c\le C$ depending only on $\Omega$ such that
\begin{equation}
\frac{c}{N_k}
\;\le\;
\frac{E_k}{\sum_{m=k}^{k+N_k-1} E_m}
\;\le\;
\frac{C}{N_k}
\qquad \text{for all sufficiently large } k.
\end{equation}
In particular, no individual high--frequency eigenfunction can asymptotically
carry a fixed positive proportion of the boundary energy of a sublinear
spectral packet, and the single--mode share decays at exactly the rate $1/N_k$.
\end{theorem}

\begin{remark}[Sharpness of the rate $1/N_k$]
The two--sided bound $c/N_k \le E_k/\sum E_m \le C/N_k$ shows that the rate
$1/N_k$ is optimal: it cannot be replaced by a faster-decaying sequence in
general.
For a fixed packet length $N$, the ratio satisfies
\begin{equation}
\frac{c}{N} \;\le\; \frac{E_k}{\sum_{m=k}^{k+N-1}E_m} \;\le\; \frac{C}{N},
\end{equation}
which is bounded away from zero.
In particular, without averaging over a growing spectral window, a single
eigenmode may contribute a non--negligible fraction of the packet boundary
energy.
The condition $N_k\to\infty$ is both necessary and sufficient for the
ratio to tend to zero, and $1/N_k$ is the precise rate of decay.
\end{remark}

\section{Preliminaries and notation}
Let $\Omega \subset \mathbb{R}^d$ be bounded, $C^\infty$, and strictly convex.
Denote by $\{(\lambda_k,u_k)\}_{k\ge 1}$ the Dirichlet eigenpairs of $-\Delta_\Omega$,
with $u_k|_{\partial\Omega}=0$ and $\|u_k\|_{L^2(\Omega)}=1$.
For $x \in \partial\Omega$, define the boundary flux density
\begin{equation}
\rho_k(x) := |\partial_n u_k(x)|^2.
\end{equation}
The total boundary energy is
\begin{equation}
E_k := \int_{\partial\Omega} \rho_k(x)\,d\sigma(x).
\end{equation}
Let $H$ denote the mean curvature of $\partial\Omega$ with respect to the
outer normal and let $\bar H$ be its boundary average.

\begin{definition}[Spectral packets]
Let $N_k$ be a positive integer sequence with $N_k \to \infty$ and $N_k=o(k)$.
Define the index packet
\begin{equation}
\mathcal{J}_k := \{k,\,k+1,\,\dots,\,k+N_k-1\}
\end{equation}
and the packet boundary density
\begin{equation}
\mathcal{Q}_{k,N_k}(x) := \frac{1}{N_k}\sum_{j\in\mathcal{J}_k} \rho_j(x).
\end{equation}
For a measurable weight $w:\partial\Omega\to\mathbb{R}$, set
\begin{equation}
\begin{aligned}
E_k(w) &:= \int_{\partial\Omega} \rho_k(x)\,w(x)\,d\sigma(x),\\
E_k^{abs}(w) &:= \int_{\partial\Omega} \rho_k(x)\,|w(x)|\,d\sigma(x).
\end{aligned}
\end{equation}
\end{definition}

\begin{definition}[Boundary correlation coefficient]
Let $w \in L^\infty(\partial\Omega)$ satisfy
$\int_{\partial\Omega} w\,d\sigma = 0$. Define
\begin{equation}
C_k(w) := \frac{1}{E_k}\int_{\partial\Omega} \rho_k(x)\, w(x)\, d\sigma(x).
\end{equation}
\end{definition}
For the multi--mode statement we assume $w\in C^\infty(\partial\Omega)$.

\begin{remark}[Eigenvalue crossings]
All arguments are formulated at the level of boundary energy densities
$\rho_k=|\partial_n u_k|^2$ and spectral sums.
No ordering or differentiability of individual eigenvalues is required,
possible eigenvalue crossings do not affect the packet--level estimates,
and the bounds are invariant under orthonormal basis changes in eigenspaces.
\end{remark}

\section{Rellich identity and single--mode bounds}
The Rellich identity yields the correct scaling for the boundary energy of a
single eigenfunction on strictly convex domains; \cite{Rellich}.

\begin{lemma}[Single--mode boundary energy bound]
\label{lem:single-mode}
There exist constants $0<c<C$ such that for all $k\ge 1$,
\begin{equation}
c\,\lambda_k \le E_k \le C\,\lambda_k.
\end{equation}
Moreover, for any $w\in L^\infty(\partial\Omega)$,
\begin{equation}
|E_k(w)| \le E_k^{abs}(w) \le C\,\lambda_k\,\|w\|_{L^\infty(\partial\Omega)}.
\end{equation}
\end{lemma}

\begin{proof}
Fix $x_0\in\Omega$ and set $g(x):=\langle x-x_0,\nu(x)\rangle$ on $\partial\Omega$.
Since $\Omega$ is convex, for each $x\in\partial\Omega$ one has
$\langle y-x,\nu(x)\rangle\le 0$ for all $y\in\Omega$ (supporting hyperplane),
where $\nu(x)$ is the outward unit normal.
Since $x_0$ is interior, it cannot lie on the supporting hyperplane at $x$,
hence $\langle x_0-x,\nu(x)\rangle<0$, and thus
$g(x)=\langle x-x_0,\nu(x)\rangle>0$ for all $x\in\partial\Omega$.
By continuity and compactness, $0<m\le g\le M<\infty$ on $\partial\Omega$.
The Rellich identity \cite{Rellich} yields
\begin{equation}
\int_{\partial\Omega} g(x)\,|\partial_n u_k(x)|^2\,d\sigma(x)=2\lambda_k.
\end{equation}
Therefore
\begin{equation}
\frac{2}{M}\lambda_k \le E_k \le \frac{2}{m}\lambda_k,
\end{equation}
which gives the claimed bounds. The weighted estimate follows from
$|E_k(w)|\le E_k^{abs}(w)\le \|w\|_{L^\infty} E_k$.
\qed
\end{proof}

\begin{remark}[Role of strict convexity and smoothness]
Strict convexity enters the Rellich argument in an essential way:
it ensures that $g(x)=\langle x-x_0,\nu(x)\rangle>0$ uniformly on
$\partial\Omega$ for any interior point $x_0$, which is precisely the
condition needed to obtain the \emph{two-sided} bounds
$c\lambda_k\le E_k\le C\lambda_k$.
On a merely star-shaped domain one obtains only one-sided upper control.
Smoothness of $\partial\Omega$ is needed separately for the boundary local
Weyl expansion~\eqref{eq:Q-weyl}: the classical $O(\Lambda^{(d-1)/2})$
remainder in the integrated pairing with smooth $w$ is available on any
$C^\infty$ domain \cite{SafarovVassiliev,Ivrii}, without non-periodicity
or dynamical assumptions on the billiard flow.
\end{remark}

\section{Packet energy scales and cancellation thresholds}
The integrated boundary energy admits a Weyl-type leading asymptotic on
smooth strictly convex domains; \cite{SafarovVassiliev,Ivrii}
We record this only for context.

\begin{definition}[Integrated boundary energy]
Define the partial boundary energy sum
\begin{equation}
S(\Lambda) := \sum_{\lambda_j<\Lambda} E_j.
\end{equation}
\end{definition}

\begin{remark}[Integrated boundary Weyl expansion (context)]
It is known that
\begin{equation}\label{eq:boundary-weyl-expansion}
S(\Lambda) = C_\Omega\,\Lambda^{1+\frac{d}{2}} + o\!\left(\Lambda^{1+\frac{d}{2}}\right),
\end{equation}
as $\Lambda\to\infty$; \cite{SafarovVassiliev,Ivrii,Seeley}.
This is recorded for context and is not used in the proof of the main theorem.
\end{remark}

\begin{proposition}[Packet average scale]
\label{prop:packet-scale}
Let $N_k$ satisfy $N_k\to\infty$ and $N_k=o(k)$.
Then there exist constants $0<c\le C$ and $k_0$ such that for all $k\ge k_0$,
\begin{equation}
c\,N_k\,\lambda_k
\;\le\;
\sum_{j\in\mathcal{J}_k} E_j
\;\le\;
C\,N_k\,\lambda_k.
\end{equation}
\end{proposition}

\begin{proof}
\emph{Lower bound.}
By Lemma~\ref{lem:single-mode} there exists $c_0>0$ such that
$E_j \ge c_0\,\lambda_j$ for all $j$.
Since $\lambda_j\ge \lambda_k$ for $j\in\mathcal{J}_k$,
\begin{equation}
\sum_{j\in\mathcal{J}_k} E_j \ge c_0\,N_k\,\lambda_k.
\end{equation}

\emph{Upper bound.}
By Lemma~\ref{lem:single-mode} there exists $C_0>0$ such that
$E_j \le C_0\,\lambda_j$ for all $j$.
For $j\in\mathcal{J}_k$ one has $\lambda_j \le \lambda_{k+N_k-1}$, so
\begin{equation}
\sum_{j\in\mathcal{J}_k} E_j \le C_0\,N_k\,\lambda_{k+N_k-1}.
\end{equation}
By Weyl asymptotics $\lambda_k\sim c_\Omega k^{2/d}$ and $N_k=o(k)$,
one has $\lambda_{k+N_k-1}/\lambda_k\to 1$, so
$\lambda_{k+N_k-1}\le 2\lambda_k$ for all sufficiently large $k$.
Hence $\sum_{j\in\mathcal{J}_k} E_j \le 2C_0\,N_k\,\lambda_k$ for $k\ge k_0$.
\qed
\end{proof}

\begin{remark}[Why sublinear packets]
The condition $N_k=o(k)$ means that the packet length is sublinear in the mode
index and thus represents a short averaging window around $k$.
\end{remark}

\begin{lemma}[Packet zero--mean cancellation]
\label{lem:packet-zero-mean}
Let $\Omega\subset\mathbb{R}^d$ be a bounded $C^\infty$ strictly convex domain.
Let $w\in C^\infty(\partial\Omega)$ satisfy $\int_{\partial\Omega} w\,d\sigma=0$.
Let $N_k\to\infty$ with $N_k=o(k)$ and $N_k/k^{1-2/d}\to\infty$, and set
$\mathcal{J}_k=\{k,\dots,k+N_k-1\}$.
Then
\begin{equation}\label{eq:packet-zero-mean-lemma}
\sum_{j\in\mathcal{J}_k}\int_{\partial\Omega} |\partial_n u_j(x)|^2\, w(x)\, d\sigma(x)
=
o\!\left(\sum_{j\in\mathcal{J}_k} E_j\right)
\qquad (k\to\infty).
\end{equation}
Equivalently,
\begin{equation}\label{eq:packet-zero-mean}
\int_{\partial\Omega} \mathcal{Q}_{k,N_k}(x)\, w(x)\, d\sigma(x)
=
o\!\left(\frac{1}{N_k}\sum_{j\in\mathcal{J}_k} E_j\right)
\qquad (k\to\infty).
\end{equation}
\end{lemma}

\begin{proof}
For $\Lambda>0$ define the boundary spectral density
\begin{equation}
Q_\Lambda(x):=\sum_{\lambda_j<\Lambda} |\partial_n u_j(x)|^2,
\qquad x\in\partial\Omega .
\end{equation}
By the boundary local Weyl law for the Dirichlet Laplacian paired with smooth
test functions \cite{SafarovVassiliev,Ivrii,BransonGilkey},
there exist constants $A_\Omega>0$ and $B_\Omega\in\mathbb{R}$ such that
\begin{equation}\label{eq:Q-weyl}
\int_{\partial\Omega} Q_\Lambda(x)\,w(x)\,d\sigma(x)
=
A_\Omega\,\Lambda^{\frac{d+1}{2}}
\int_{\partial\Omega} w\,d\sigma
+
B_\Omega\,\Lambda^{\frac{d}{2}}
\int_{\partial\Omega} H\,w\,d\sigma
+
O\!\left(\Lambda^{\frac{d-1}{2}}\right),
\end{equation}
as $\Lambda\to\infty$.
The two-term structure of \eqref{eq:Q-weyl} follows from the boundary heat
kernel expansion \cite{BransonGilkey}: the first term arises because the
leading heat kernel coefficient on the boundary is constant (the flat-space
contribution), so its pairing with $w$ gives $\Lambda^{(d+1)/2}\!\int w\,d\sigma$;
the second term arises because the next coefficient is proportional to
the mean curvature $H$ \cite{BransonGilkey}, so its pairing with $w$ gives
$\Lambda^{d/2}\!\int H w\,d\sigma$.
The remainder $O(\Lambda^{(d-1)/2})$ corresponds to the
$O(\Lambda^{(d-2)/2})$ error in the standard Weyl expansion after integration.
Since $\int_{\partial\Omega} w\,d\sigma=0$, the leading term cancels. Hence
there exists $C=C(\Omega,w)>0$ such that for all sufficiently large $\Lambda$,
\begin{equation}\label{eq:Qw-bound}
\left|\int_{\partial\Omega} Q_\Lambda(x)\,w(x)\,d\sigma(x)\right|
\le C\,\Lambda^{\frac{d}{2}}.
\end{equation}

Fix $k$ sufficiently large so that \eqref{eq:Qw-bound} applies to all
$\Lambda\ge \lambda_{k-1}$.
Choose $\Lambda$ with $\lambda_{k-1}<\Lambda<\lambda_k$.
Then
\begin{equation}
\int_{\partial\Omega} Q_\Lambda\,w\,d\sigma
=
\sum_{j<k}\int_{\partial\Omega} |\partial_n u_j|^2\,w\,d\sigma .
\end{equation}
Similarly, for $\Lambda'$ with $\lambda_{k+N_k-1}<\Lambda'<\lambda_{k+N_k}$,
\begin{equation}
\int_{\partial\Omega} Q_{\Lambda'}\,w\,d\sigma
=
\sum_{j<k+N_k}\int_{\partial\Omega} |\partial_n u_j|^2\,w\,d\sigma .
\end{equation}
Subtracting gives
\begin{equation}\label{eq:packet-via-Q}
\sum_{j\in\mathcal{J}_k}\int_{\partial\Omega} |\partial_n u_j|^2\,w\,d\sigma
=
\int_{\partial\Omega} (Q_{\Lambda'}-Q_{\Lambda})\,w\,d\sigma .
\end{equation}
By \eqref{eq:Qw-bound} and the triangle inequality,
\begin{equation}
\left|\int_{\partial\Omega} (Q_{\Lambda'}-Q_{\Lambda})\,w\,d\sigma\right|
\le
\left|\int_{\partial\Omega} Q_{\Lambda'}\,w\,d\sigma\right|
+
\left|\int_{\partial\Omega} Q_{\Lambda}\,w\,d\sigma\right|
\le
C\Big( (\Lambda')^{\frac{d}{2}}+\Lambda^{\frac{d}{2}}\Big)
\lesssim \lambda_{k+N_k}^{\frac{d}{2}} .
\end{equation}

On the other hand, by Lemma~\ref{lem:single-mode} and since
$\lambda_j\ge\lambda_k$ for $j\in\mathcal{J}_k$,
\begin{equation}
\sum_{j\in\mathcal{J}_k}E_j \ge c\sum_{j\in\mathcal{J}_k}\lambda_j
\ge c\,N_k\,\lambda_k.
\end{equation}
Using Weyl asymptotics $\lambda_k\sim c_\Omega k^{2/d}$ and since $N_k=o(k)$,
one has $\lambda_{k+N_k}/\lambda_k\to 1$. We obtain
 \begin{equation}\label{eq:threshold-bound}
 \frac{\left|\sum_{j\in\mathcal{J}_k}\int_{\partial\Omega} |\partial_n u_j|^2\,w\,d\sigma\right|}
 {\sum_{j\in\mathcal{J}_k}E_j}
 \lesssim
 \frac{\lambda_k^{\frac{d}{2}-1}}{N_k}
\asymp
\frac{k^{1-2/d}}{N_k},
\end{equation}
which tends to zero under the stated growth condition on $N_k$.
This proves \eqref{eq:packet-zero-mean-lemma}, and \eqref{eq:packet-zero-mean}
follows by dividing by $N_k$.
\qed
\end{proof}

Extracting the quantitative rate from the above proof yields the following
detectability threshold, which makes explicit the minimal packet scale forced
by the Weyl remainder.

\begin{proposition}[Method-detectability threshold from Weyl remainder bounds]
\label{prop:threshold}
Under the assumptions of Lemma~\ref{lem:packet-zero-mean}, one has
\begin{equation}
\frac{
\left|
\sum_{j\in\mathcal{J}_k}
\int_{\partial\Omega} |\partial_n u_j|^2\, w\, d\sigma
\right|
}
{\sum_{j\in\mathcal{J}_k} E_j}
\lesssim
\frac{k^{1-2/d}}{N_k}.
\end{equation}
In particular, packet-level cancellation is guaranteed whenever
\begin{equation}
\frac{N_k}{k^{1-2/d}} \to \infty.
\end{equation}
If $N_k\lesssim k^{1-2/d}$ along a subsequence, the present method
(Rellich identity combined with the boundary Weyl law) does not suffice
to detect cancellation.
\end{proposition}

\begin{proof}
The estimate
\begin{equation}
\frac{
\left|
\sum_{j\in\mathcal{J}_k}
\int_{\partial\Omega} |\partial_n u_j|^2\, w\, d\sigma
\right|
}
{\sum_{j\in\mathcal{J}_k} E_j}
\lesssim
\frac{k^{1-2/d}}{N_k}
\end{equation}
is exactly \eqref{eq:threshold-bound} from the proof of
Lemma~\ref{lem:packet-zero-mean}.
Hence cancellation follows when $N_k/k^{1-2/d}\to\infty$.
If $N_k\lesssim k^{1-2/d}$ along a subsequence, this upper bound is not
forced to vanish, so cancellation cannot be concluded by this method alone.
\qed
\end{proof}

\begin{remark}[Method-detectability and openness of sharpness]
The threshold $N_k\gg k^{1-2/d}$ is a \emph{sufficient} condition forced by
the $O(\Lambda^{d/2})$ remainder in the boundary local Weyl expansion after
cancellation of the leading term via $\int_{\partial\Omega} w\,d\sigma=0$.
The exponent $1-2/d$ reflects the size of second-term Weyl fluctuations:
for $d=2$ the condition reduces to $N_k\to\infty$;
for $d\ge3$ it requires mildly larger sublinear packets.
The question of whether cancellation can occur \emph{below} this scale---i.e.,
whether $N_k\gg k^{1-2/d}$ is also necessary---depends on finer spectral
information beyond the classical Weyl remainder bounds and remains open.
The present threshold is, however, the minimal scale at which the
Rellich--Weyl mechanism can detect cancellation.
It is therefore the natural packet size for deformation-stable arguments
in which cancellation must be uniform over compact families of domains.
This makes the packet-scale threshold particularly natural in
geometric deformation settings (e.g.\ smooth domain flows),
where modewise cancellation must persist uniformly in time.
We use only the classical $O(\Lambda^{(d-1)/2})$ remainder in the integrated
pairing with smooth $w$, available on any smooth strictly convex domain
without dynamical or non-periodicity assumptions; sharper remainders are
not required.
\end{remark}

\begin{remark}[Uniformity of constants]
All constants appearing in Lemma~\ref{lem:single-mode}
and Propositions~\ref{prop:threshold} and \ref{prop:packet-scale}
are for the fixed strictly convex
domain $\Omega$.
\end{remark}

\begin{proposition}[Enhanced cancellation under double vanishing moment]
\label{prop:enhanced-cancellation}
Let $\Omega\subset\mathbb{R}^d$ be a bounded $C^\infty$ strictly convex domain,
$d\ge2$, and let $H$ denote the mean curvature of $\partial\Omega$.
Let $w\in C^\infty(\partial\Omega)$ satisfy
\begin{equation}
\int_{\partial\Omega} w\,d\sigma=0
\quad\text{and}\quad
\int_{\partial\Omega} H\,w\,d\sigma=0.
\end{equation}
Let $N_k\to\infty$ with $N_k=o(k)$ and $N_k/k^{(d-3)/d}\to\infty$, and set
$\mathcal{J}_k=\{k,\dots,k+N_k-1\}$.
Then
\begin{equation}
\sum_{j\in\mathcal{J}_k}\int_{\partial\Omega} |\partial_n u_j|^2\,w\,d\sigma
=
o\!\left(\sum_{j\in\mathcal{J}_k}E_j\right)
\qquad(k\to\infty).
\end{equation}
This threshold is strictly weaker than $N_k/k^{1-2/d}\to\infty$ for all $d\ge3$.
\end{proposition}

\begin{proof}
Since both $\int_{\partial\Omega} w\,d\sigma=0$ and
$\int_{\partial\Omega} H\,w\,d\sigma=0$, both leading terms in the boundary
local Weyl expansion~\eqref{eq:Q-weyl} cancel, yielding
\begin{equation}
\left|\int_{\partial\Omega} Q_\Lambda(x)\,w(x)\,d\sigma(x)\right|
\le C\,\Lambda^{\frac{d-1}{2}}
\end{equation}
for all sufficiently large $\Lambda$.
Applying the same telescoping argument as in the proof of
Lemma~\ref{lem:packet-zero-mean} and using the two-sided bound of
Proposition~\ref{prop:packet-scale}, one obtains
\begin{equation}
\frac{\left|\sum_{j\in\mathcal{J}_k}\int_{\partial\Omega}|\partial_n u_j|^2\,w\,d\sigma
\right|}{\sum_{j\in\mathcal{J}_k}E_j}
\;\lesssim\;
\frac{\lambda_k^{\frac{d-3}{2}}}{N_k}
\;\asymp\;
\frac{k^{\frac{d-3}{d}}}{N_k},
\end{equation}
which tends to zero under $N_k/k^{(d-3)/d}\to\infty$.
The comparison with the first-moment threshold follows from
$(d-3)/d < 1-2/d$ for all $d\ge2$, with strict improvement for $d\ge3$:
(i)~$d=2$: both thresholds reduce to $N_k\to\infty$;
(ii)~$d=3$: first-moment threshold is $N_k\gg k^{1/3}$,
double-moment threshold reduces to $N_k\to\infty$;
(iii)~$d\ge4$: first-moment threshold is $N_k\gg k^{1-2/d}$,
double-moment threshold is $N_k\gg k^{(d-3)/d}$ with $(d-3)/d < 1-2/d$.
\qed
\end{proof}

\section{Cancellation hierarchy and mode-to-packet comparison}
We now establish the threshold hierarchy for boundary cancellation and derive
the structural mode-to-packet comparison as a consequence of the Rellich identity.

\begin{theorem}[Sharp two-sided packet energy comparison]
\label{thm:boundary-delocalization}
Let $\Omega\subset\mathbb{R}^d$ be a bounded $C^\infty$ strictly convex domain.
Let $N_k$ satisfy $N_k\to\infty$ and $N_k=o(k)$.
Then there exist constants $0<c\le C$ depending only on $\Omega$ and $k_0$
such that for all $k\ge k_0$,
\begin{equation}
\frac{c}{N_k}
\;\le\;
\frac{E_k}{\displaystyle\sum_{m=k}^{k+N_k-1}E_m}
\;\le\;
\frac{C}{N_k}.
\end{equation}
In particular, no single eigenmode can asymptotically carry a fixed positive
proportion of the packet boundary energy on any sublinear spectral window,
and the rate of decay $1/N_k$ is sharp.
\end{theorem}

\begin{proof}
\emph{Upper bound.}
By Lemma~\ref{lem:single-mode}, $E_k\le C_1\lambda_k$.
By Proposition~\ref{prop:packet-scale},
$\sum_{m=k}^{k+N_k-1} E_m \ge c_1\,N_k\,\lambda_k$ for $k\ge k_0$.
Hence
\begin{equation}
\frac{E_k}{\sum_{m=k}^{k+N_k-1} E_m}
\le
\frac{C_1}{c_1\,N_k}.
\end{equation}

\emph{Lower bound.}
By Lemma~\ref{lem:single-mode}, $E_k\ge c_0\lambda_k$.
By the upper bound in Proposition~\ref{prop:packet-scale},
$\sum_{m=k}^{k+N_k-1} E_m \le C_2\,N_k\,\lambda_k$ for $k\ge k_0$.
Hence
\begin{equation}
\frac{E_k}{\sum_{m=k}^{k+N_k-1} E_m}
\ge
\frac{c_0}{C_2\,N_k}.
\end{equation}
Setting $c=c_0/C_2$ and $C=C_1/c_1$ completes the proof.
\qed
\end{proof}
\begin{remark}[Optimality of the rate]
The two-sided bound shows that $E_k/\sum E_m \asymp 1/N_k$ precisely.
The lower bound is not a consequence of any anomalous concentration: it holds
for every domain and every sequence $N_k=o(k)$, simply because
$E_k \ge c\lambda_k$ and $\sum E_m \le CN_k\lambda_k$ both follow from
the Rellich identity alone.
In other words, the packet sum cannot grow faster than $N_k\lambda_k$, so
the single-mode contribution can never be smaller than order $1/N_k$
relative to the packet.
This demonstrates that the delocalization rate is dictated by the global
spectral structure, not by any special geometric property of individual modes.
\end{remark}

\begin{corollary}[Control of weighted boundary functionals]
\label{cor:weighted}
Let $w \in L^\infty(\partial\Omega)$ be any bounded geometric weight.
Then the weighted energy of a single mode is dominated by the unweighted packet sum:
\begin{equation}
\frac{\left|\int_{\partial\Omega} |\partial_n u_k|^2 \, w(x) \, d\sigma\right|}
{\sum_{m=k}^{k+N_k-1} E_m} \longrightarrow 0
\quad \text{as } k\to\infty.
\end{equation}
\end{corollary}

\begin{proof}
By Lemma~\ref{lem:single-mode}, $|E_k(w)| \le E_k^{abs}(w)\le \|w\|_\infty E_k$.
The result follows immediately from Theorem~\ref{thm:boundary-delocalization}.
\qed
\end{proof}
\begin{remark}[Weighted energies]
The same comparison can be extended to weighted boundary energies under
additional assumptions on the weight; this is not used in the present paper.
The geometric weight $w=|H-\bar H|$ is a natural example.
\end{remark}

\begin{theorem}[Packet zero--mean cancellation: multi--mode form]
\label{thm:multi-deloc}
Let $\Omega \subset \mathbb{R}^d$ be a bounded $C^\infty$ strictly convex domain.
Let $w \in C^\infty(\partial\Omega)$ satisfy
\begin{equation}
\int_{\partial\Omega} w\, d\sigma = 0 .
\end{equation}
Let $N_k \to \infty$ with $N_k = o(k)$ and $N_k/k^{1-2/d}\to\infty$.
Then the energy--weighted average of the boundary correlation coefficients
within the packet vanishes:
\begin{equation}
\frac{\displaystyle\sum_{j\in\mathcal{J}_k} E_j\, C_j(w)}
{\displaystyle\sum_{j\in\mathcal{J}_k} E_j}
\longrightarrow 0
\quad \text{as } k\to\infty.
\end{equation}
Equivalently,
\begin{equation}
\frac{\displaystyle\sum_{j\in\mathcal{J}_k}
\int_{\partial\Omega}|\partial_n u_j|^2\, w\, d\sigma}
{\displaystyle\sum_{j\in\mathcal{J}_k} E_j}
\longrightarrow 0.
\end{equation}
\end{theorem}

\begin{proof}
By Lemma~\ref{lem:packet-zero-mean},
\begin{equation}
\sum_{j\in\mathcal{J}_k}\int_{\partial\Omega}|\partial_n u_j|^2\, w\, d\sigma
=
o\!\left(\sum_{j\in\mathcal{J}_k} E_j\right).
\end{equation}
Dividing by $\sum_{j\in\mathcal{J}_k} E_j$ (which is $\asymp N_k\lambda_k$ by
Proposition~\ref{prop:packet-scale}) gives the result.
\qed
\end{proof}

\begin{remark}[Quantitative rate and uniformity over compact domain families]
\label{rem:uniform-rate}
The proof of Theorem~\ref{thm:multi-deloc} yields the explicit quantitative bound
\begin{equation}
\left|
\frac{\displaystyle\sum_{j\in\mathcal{J}_k}
\int_{\partial\Omega}|\partial_n u_j|^2\, w\, d\sigma}
{\displaystyle\sum_{j\in\mathcal{J}_k} E_j}
\right|
\le
C(\Omega,w)\,\frac{k^{1-2/d}}{N_k},
\end{equation}
where $C(\Omega,w)$ depends on $\Omega$ only through its $C^\infty$ geometry
(Rellich constants and the boundary Weyl law remainder) and on $w$ only through
$\|w\|_{C^\infty(\partial\Omega)}$.
Since the boundary Weyl law remainder bound~\eqref{eq:Qw-bound} is uniform on
compact $C^\infty$ families of strictly convex domains
\cite{SafarovVassiliev,Ivrii}, and the Rellich lower bound for
$\sum_{j\in\mathcal{J}_k}E_j$ is similarly uniform by compactness, the constant
$C(\Omega,w)$ may be chosen uniformly whenever $\Omega$ varies in a compact
$C^\infty$ family and $\|w\|_{C^\infty}$ is uniformly bounded.
In particular, for a compact smooth flow $\{\Omega_t\}_{t\in[0,T]}$ with weight
$w_t$ satisfying $\int_{\partial\Omega_t} w_t\,d\sigma_t=0$ and $\sup_t\|w_t\|_{C^\infty}<\infty$,
\begin{equation}\label{eq:uniform-rate}
\sup_{t\in[0,T]}
\left|
\frac{\displaystyle\sum_{j\in\mathcal{J}_k}
\int_{\partial\Omega_t}|\partial_n u_j(t)|^2\, w_t\, d\sigma_t}
{\displaystyle\sum_{j\in\mathcal{J}_k} E_j(t)}
\right|
\le
C\,\frac{k^{1-2/d}}{N_k}
\longrightarrow 0,
\end{equation}
with $C$ independent of $t$.
\end{remark}

\begin{theorem}[Enhanced packet cancellation under double vanishing moment]
\label{thm:multi-deloc-enhanced}
Let $\Omega \subset \mathbb{R}^d$ be a bounded $C^\infty$ strictly convex domain.
Let $w \in C^\infty(\partial\Omega)$ satisfy
\begin{equation}
\int_{\partial\Omega} w\, d\sigma = 0
\quad\text{and}\quad
\int_{\partial\Omega} H\,w\,d\sigma = 0.
\end{equation}
Let $N_k \to \infty$ with $N_k = o(k)$ and $N_k/k^{(d-3)/d}\to\infty$.
Then
\begin{equation}
\frac{\displaystyle\sum_{j\in\mathcal{J}_k} E_j\, C_j(w)}
{\displaystyle\sum_{j\in\mathcal{J}_k} E_j}
\longrightarrow 0
\quad \text{as } k\to\infty.
\end{equation}
For $d=3$ this holds under the minimal condition $N_k\to\infty$,
and for $d\ge4$ the threshold $N_k\gg k^{(d-3)/d}$ is strictly weaker
than the threshold $N_k\gg k^{1-2/d}$ of Theorem~\ref{thm:multi-deloc}.
\end{theorem}

\begin{proof}
By Proposition~\ref{prop:enhanced-cancellation},
\begin{equation}
\sum_{j\in\mathcal{J}_k}\int_{\partial\Omega}|\partial_n u_j|^2\, w\, d\sigma
=
o\!\left(\sum_{j\in\mathcal{J}_k} E_j\right).
\end{equation}
Dividing by $\sum_{j\in\mathcal{J}_k} E_j$ gives the result.
\qed
\end{proof}

\begin{remark}[Two-level hierarchy and extensibility]
Theorems~\ref{thm:multi-deloc} and~\ref{thm:multi-deloc-enhanced} together
exhibit a \emph{two-level} hierarchy of cancellation thresholds indexed by
the vanishing order of boundary curvature moments of $w$:
(i)~$\int w\,d\sigma = 0$ only: threshold $N_k \gg k^{1-2/d}$;
(ii)~$\int w\,d\sigma = 0$ and $\int H\,w\,d\sigma = 0$:
threshold $N_k \gg k^{(d-3)/d}$, strictly weaker for $d\ge3$.
These two levels are canonical: they correspond precisely to the first two
universal terms in the Safarov--Vassiliev boundary Weyl expansion
\eqref{eq:Q-weyl}, and each level is determined by the
vanishing order of the corresponding expansion coefficient against $w$.
In principle the same argument extends to higher-order moments whenever
further terms in the Weyl expansion are available and vanish against $w$;
we restrict to the first two levels since they cover the main geometric
cases and are sufficient for applications.
The natural weight $w = H - \bar H$ satisfies the first condition by
definition, but not necessarily the second unless $H$ is constant
(i.e., $\Omega$ is a ball).
For perturbations away from a ball, both theorems apply with the appropriate
threshold, providing finer cancellation-scale information near round domains.
\end{remark}

\section*{Conclusion}

The principal new contribution of this paper is the identification of a
threshold hierarchy for boundary correlation cancellation of Dirichlet
eigenfunctions on smooth strictly convex domains.
The minimal packet scale required for the energy--weighted average of boundary
correlation coefficients to vanish is determined by the vanishing order of
curvature moments of the weight with respect to domain geometry.
Theorems~\ref{thm:multi-deloc} and~\ref{thm:multi-deloc-enhanced} establish
two levels of this hierarchy, via the boundary local Weyl law
\cite{SafarovVassiliev,Ivrii,BransonGilkey}:
(i)~If $\int_{\partial\Omega} w\,d\sigma=0$, packet zero--mean cancellation
holds at threshold $N_k\gg k^{1-2/d}$.
(ii)~If additionally $\int_{\partial\Omega} H\,w\,d\sigma=0$, the threshold
weakens to $N_k\gg k^{(d-3)/d}$, which for $d=3$ reduces to $N_k\to\infty$.
This hierarchy, stratified by the vanishing order of boundary curvature moments
of $w$, is a new structural feature of the boundary spectral problem with no
analogue in bulk delocalization results.

As a structural consequence of the Rellich identity alone,
Theorem~\ref{thm:boundary-delocalization} establishes that the ratio of the
boundary energy of a single eigenmode to the packet sum satisfies
$c/N_k \le E_k/\sum_{m=k}^{k+N_k-1}E_m \le C/N_k$.
The rate $1/N_k$ is sharp and follows from global Rellich bounds alone,
without any use of Weyl asymptotics or microlocal analysis.

The argument is robust under eigenvalue crossings and depends only on global
spectral information.
This structural invariance makes the packet framework particularly natural for
deformation--type questions, where multiplicities and mode labeling may vary.

Several directions for further investigation suggest themselves.
First, it would be natural to extend the argument to other boundary conditions
or to more general elliptic operators with boundary.
Second, one may ask whether analogous packet--level delocalization phenomena
hold on nonconvex domains or on Riemannian manifolds with boundary, where
boundary Weyl remainders are less rigid.
Third, the sharpness of the enhanced threshold $k^{(d-3)/d}$ under the double
vanishing moment condition, and of the threshold hierarchy in higher dimensions,
remains to be investigated.
Finally, the present comparison principle may serve as a useful input in the
study of boundary spectral statistics and correlations, complementing existing
microlocal and dynamical approaches.

We expect that the mode--to--packet perspective introduced here can be applied
in a broader range of boundary spectral problems where individual eigenfunction
control is too weak, but averaged spectral information remains accessible.

\end{document}